\documentclass[english]{article}
\usepackage[T1]{fontenc}
\usepackage[latin9]{inputenc}
\usepackage{geometry}
\usepackage{textcomp}
\usepackage{amsmath}
\usepackage{amsthm}
\usepackage{amssymb}
\usepackage{listings}
\usepackage{color}
\usepackage{chngcntr}
\usepackage{todonotes}
\counterwithin{table}{section}

\makeatletter
\theoremstyle{plain}
\newtheorem{thm}{\protect\theoremname}
  \theoremstyle{definition}

  \theoremstyle{plain}
  \newtheorem{conjecture}[thm]{\protect\conjecturename}
  \theoremstyle{plain}
  \newtheorem{lem}[thm]{\protect\lemmaname}
\newtheorem*{lem*}{\protect\lemmaname}
  \newtheorem{ques}[thm]{Question}
\newenvironment{namedlemma*}[1]
  {\renewcommand{\lemmaname}{#1}%
   \begin{lem*}}
  {\end{lem*}}

\usepackage{babel}
\providecommand{\conjecturename}{Conjecture}
  \providecommand{\definitionname}{Definition}
  \providecommand{\lemmaname}{Lemma}
\providecommand{\theoremname}{Theorem}

\DeclareFixedFont{\ttb}{T1}{txtt}{bx}{n}{12} 
\DeclareFixedFont{\ttm}{T1}{txtt}{m}{n}{12}  

\numberwithin{equation}{section}

\makeatother

\begin{document}

\title{A proof of Tomescu's graph coloring conjecture}

\author{Jacob Fox\thanks{Department of Mathematics, Stanford University, Stanford, CA 94305, USA. Email: {\tt jacobfox@stanford.edu}. Research supported by a Packard Fellowship, and by NSF Career Award DMS-1352121.},\and Xiaoyu He\thanks{Department of Mathematics, Stanford University, Stanford, CA 94305, USA. Email: {\tt alkjash@stanford.edu}.},\and Freddie Manners\thanks{Department of Mathematics, Stanford University, Stanford, CA 94305, USA. Email: {\tt fmanners@stanford.edu}.}}
\maketitle
\begin{abstract}
In 1971, Tomescu conjectured that every connected graph $G$ on $n$
vertices with chromatic number $k\geq4$ has at most $k!(k-1)^{n-k}$
proper $k$-colorings. Recently, Knox and Mohar proved Tomescu's conjecture for $k=4$ and $k=5$. In this paper, we complete the proof of Tomescu's conjecture for all $k\ge 4$, and show that equality occurs if and only if $G$ is a $k$-clique with trees attached to each vertex. 
\end{abstract}

\section{Introduction}

Let $k$ be a positive integer and $G=(V,E)$ be a graph\footnote{All graphs in this paper are simple and undirected.}.
A {\it proper $k$-coloring}, or simply a {\it $k$-coloring}, of $G$ is a function
$c:V\to[k]$ (here $[k] = \{1,\ldots, k\}$) such that $c(u)\ne c(v)$ whenever
$uv\in E$. The chromatic number $\chi(G)$ is the minimum $k$ for which there exists a $k$-coloring of $G$. Let $P_{G}(k)$ denote the number of $k$-colorings
of $G$. This function is a polynomial in $k$ and is thus called the
\textit{chromatic polynomial} of $G$. In 1912, Birkhoff \cite{Bi1}
introduced the chromatic polynomial for planar graphs in an attempt
to solve the Four Color Problem using tools from analysis. The chromatic
polynomial was later defined and studied for general graphs by Whitney
\cite{Wh}.

Despite a great deal of attention over the past century, our understanding of the chromatic polynomial is still quite
poor. In particular, proving general bounds on the chromatic polynomial remains a major challenge. The first result of this type, due to Birkhoff \cite{Bi2}, states that
\[
P_{G}(k)\ge k(k-1)(k-2)(k-3)^{n-3}
\]
for any planar graph
$G$ on $n$ vertices and any real number $k\ge5$. Birkhoff and Lewis
\cite{BiLe} later conjectured this to be true for all real $k\geq4$.
The case $k=4$ is equivalent to the celebrated Four Color Theorem, which was
resolved much later.

In 1971, Tomescu \cite{To} conjectured that 
\begin{equation}\label{eq:tomescu-conj}
P_{G}(k)\leq k!(k-1)^{n-k}
\end{equation}
for every connected graph $G$ on $n$ vertices with chromatic
number $k\geq4$. This inequality can be easily shown if $G$ contains a $k$-clique. It is also tight if the $2$-core of $G$ is exactly a $k$-clique, that is, if after repeatedly deleting vertices of degree at most one, the remaining graph is a $k$-clique. Equivalently, such a graph $G$ is formed by attaching a tree to each vertex of the complete graph $K_k$. Tomescu further conjectured that such graphs are the only examples for which inequality (\ref{eq:tomescu-conj}) is tight. 

Tomescu's conjecture has received considerable attention over the past 46 years. Tomescu observed that if $k=2$, inequality (\ref{eq:tomescu-conj}) holds as an equality for every connected bipartite graph, whereas it is false for $k=3$ when $G$ is an odd cycle of length at least five. This shows that the condition $k\ge 4$ is necessary; see \cite{KnMo1,To72,To94,To97} for details. Tomescu \cite{To1} also proved his conjecture for $4$-chromatic planar graphs. Further partial results were obtained in \cite{BE15,Er1,Er2}, and the book \cite{DKT} gives an overview of the area. Recently, Knox and Mohar solved the cases $k=4$ \cite{KnMo1} and $k=5$ \cite{KnMo2} of Tomescu's conjecture.

In this paper, we prove the full conjecture.

\begin{thm}
\label{thm:main}If $G$ is a connected graph on $n$ vertices with $\chi(G)=k \geq 4$,
then 
\begin{equation*}
P_{G}(k)\le k!(k-1)^{n-k},
\end{equation*}
with equality if and only if the $2$-core of $G$ is a $k$-clique.
\end{thm}

In order to prove Theorem \ref{thm:main}, we introduce a method for bounding the number of $k$-colorings of a graph using what we call the Overprediction Lemma. The idea is to construct a probability distribution on the set of colorings by randomly coloring one vertex at a time. From this distribution one obtains an upper bound on $P_G(k)$ using the fact that the cross-entropy between two random variables is at least the entropy of either (this is equivalent to the nonnegativity of Kullback--Leibler divergence).  We expect this lemma to have further applications in graph coloring. The argument closely resembles what is sometimes called the ``entropy method'' for combinatorial upper bounds, first developed by Radhakrishnan \cite{Ra} to give an alternative proof of Br\'egman's Theorem, and applied subsequently in various forms, see e.g.\ \cite{GlLu,Ka,LiLu}.

In the next section, we state and prove the Overprediction Lemma using the entropy inequality mentioned above and develop probabilistic tools for applying it.

In Section \ref{sec:Setup}, we collect some basic results on graphs with chromatic number $k$, and set up the pieces of our inductive proof of Theorem \ref{thm:main}. There, we also show that Theorem \ref{thm:main} holds for small graphs $G$ in the range $|V| \le 2k-2$, which serve as base cases for the induction.

In Section \ref{sec:A-Linear-Program}, we describe a simple linear
program whose optimum describes an upper bound for the number of proper $k$-colorings of a possible counterexample to Tomescu's conjecture, and complete the argument in Section \ref{sec:Radiant-Vertices}. Together with the Overprediction Lemma, this proves Theorem \ref{thm:main} for large graphs $G$, in the range $|V| > 2k-2$ when $k \ge 5$. 

Finally, in Section~\ref{sec:small-4-crit} we use a more delicate calculation to handle the case $k=4$.

\section{The Overprediction Lemma\label{sec:Overprediction-Lemma}}

Fix a graph $G=(V,E)$ of chromatic number at most $k$. Given a $k$-coloring $c$ of $G$ and a (linear) ordering $\pi$ of $V$, we can predict the total number $P_G(k)$ of $k$-colorings of $G$ as follows.

Define a prefix of $\pi$ to be a set $U\subseteq V$ with the property that if $v\in U$ then any vertex before $v$ lies in $U$. Define a {\it $\pi$-precoloring} of $G$ to be a $k$-coloring of an induced subgraph of $G$ whose vertices form a prefix of $\pi$. Define a {\it maximal $\pi$-precoloring} to be a $\pi$-precoloring which cannot be extended to the next vertex in $\pi$. In particular, all $k$-colorings are also maximal $\pi$-precolorings. 

Define the {\it backneighborhood} of a vertex $v$ with respect to $\pi$, denoted $N_{\pi}^{-}(v)$, to be the set of neighbors of $v$ coming before $v$ in the ordering $\pi$. If $c$ is a maximal $\pi$-precoloring of $G$, let $X_{\pi}(c,v)$ be the number of colors in $[k]$ that do not appear among the colored vertices in $N_{\pi}^{-}(v)$. If $U$ is the set of vertices colored by $c$, let
\[
X_{\pi}(c)=\prod_{v\in U}X_{\pi}(c,v).
\]

If the vertices of $G$ are greedily colored in the order of $\pi$, the number of choices of the color for $v$, having already colored all the preceding vertices according to $c$, is exactly $X_{\pi}(c,v)$. It is thus natural to suspect that if $c$ is a $k$-coloring of $G$, then $X_{\pi}(c)$ is a good estimate for $P_G(k)$. For example, if $G = K_k$, then for any ordering $\pi$ and any $k$-coloring $c$, $P_G(k)=X_{\pi}(c)=k!$. 

We will be taking expectations of random variables over random choices of a vertex, a coloring, and an ordering of $V$. The subscripts in the expectation operator $\mathbb{E}[\cdot]$ denote the random choice(s) we are taking expectation over, and choices are made uniformly and independently from among all possible ones.

For example, we write $\mathbb{E}_{\pi,c}[X]$ for the expectation of the random
variable $X$ over a uniform random choice of a vertex ordering $\pi$
and a $k$-coloring $c$ of $G$.

\begin{namedlemma*}{The Overprediction Lemma}
\label{lem:overprediction} For any graph $G = (V,E)$ with $\chi(G)\le k$ and any ordering $\pi$ of $V$,
\[
P_G(k) \le \exp(\mathbb{E}_{c}[\log X_{\pi}(c)]).
\]
\end{namedlemma*}
\begin{proof}
Let $C$ be the family of all $k$-colorings of $G$ and $C_\pi$ be the family of all maximal $\pi$-precolorings of $G$, so $C \subseteq C_{\pi}$ and $|C|=P_G(k)$. We wish to show
\begin{equation}\label{eq:target}
\mathbb{E}_{c\in C} [\log X_\pi (c)] \ge \log |C|.
\end{equation}

Define two probability distributions $p$, $q$ over $C_\pi$ as follows. The distribution $p$ is uniform over colorings $c\in C$. The distribution $q$ picks each $c\in C_\pi$ with probability $X_\pi(c)^{-1}$. Equivalently, $q$ constructs a random maximal pre-coloring $c$ by greedily coloring $G$ one vertex at a time in the order of $\pi$, making a uniform choice for the color of $v$ at each step out of all $X_\pi(c,v)$ possible colors. This process stops at a maximal $\pi$-precoloring when it runs out of choices or vertices. 

By the nonnegativity of Kullback--Leibler divergence from $p$ to $q$, we find that

\begin{equation}\label{eq:entropy}
- \sum_{c\in C_{\pi}} p(c) \log q(c) \ge - \sum_{c\in C_{\pi}} p(c) \log p(c).
\end{equation}

Note that $p$ is supported on $C$. Thus, the left hand side of the inequality is
\begin{equation*}
- \sum_{c\in C_{\pi}} p(c) \log q(c) = \mathbb{E}_{c\in C}[\log X_\pi (c)],
\end{equation*}
while the right hand side is
\begin{equation*}
- \sum_{c\in C_{\pi}} p(c) \log p(c) = \log |C|,
\end{equation*}
so (\ref{eq:entropy}) is exactly equivalent to the desired inequality (\ref{eq:target}).
\end{proof}

Let 
\begin{equation}\label{eq:defT}
T(c,v)=\exp(\mathbb{E}_{\pi}[\log X_{\pi}(c,v)])
\end{equation}
be the geometric mean of $X_{\pi}(c,v)$ over a uniform random permutation
$\pi$, and define
\[
T=\prod_{v\in V}\exp(\mathbb{E}_{c}[\log T(c,v)]).
\]

By linearity of expectation,
\begin{align*}
\exp(\mathbb{E}_{\pi,c}[\log X_{\pi}(c)]) & = \prod_{v\in V}\exp(\mathbb{E}_{\pi,c}[\log X_{\pi}(c,v)])\\
 & = \prod_{v\in V}\exp(\mathbb{E}_{c}[\log T(c,v)])\\
 & = T.
\end{align*}

After averaging over orderings $\pi$, we can upper bound $P_G(k)$ by the quantity $T$.

Given a $k$-coloring $c$, define
$c_{i}(v)$ to be the number of neighbors of $v$ colored with color
$i$. The value of $T(c,v)$ is completely determined by the $k$-tuple $(c_{i}(v))_{i\le k}$, and we can obtain an upper bound
for $T(c,v)$ via the AM--GM inequality. Define
\begin{equation}\label{eq:defW}
W(c,v)=\sum_{i\le k} \frac{1}{c_{i}(v)+1}.
\end{equation}

\begin{lem}\label{lem:AMGM}
With $T(c,v)$ and $W(c,v)$ defined by (\ref{eq:defT}) and (\ref{eq:defW}) respectively, we have
\[
T(c,v)\le W(c,v)
\]
for every $k$-coloring $c$ and vertex $v$.
\end{lem}
\begin{proof}
Note that $X_{\pi}(c,v)=k-Z_{\pi}(c,v)$, where $Z_{\pi}(c,v)$ is the
number of colors of $c$ that appear in $N^{-}_{\pi}(v)$.
By the AM--GM inequality and linearity of expectation, 
\begin{align*}
T(c,v) & = \exp(\mathbb{E}_{\pi}[\log X_{\pi}(c,v)])\\
 & \le \mathbb{E}_{\pi}[X_{\pi}(c,v)]\\
 & = k-\mathbb{E}_{\pi}[Z_{\pi}(c,v)]\\
 & = k-\sum_{i\le k}\mathbb{E}_{\pi}[Z_{\pi}^{i}(c,v)],
\end{align*}
where $Z_{\pi}^{i}(c,v)$ is the indicator random variable for the event that color $i$ appears in $N^{-}_{\pi}(v)$. 

But for each individual color $i$, $\Pr[Z_{\pi}^{i}(c,v)]$ is 
the probability that $v$ comes after at least one of the $c_{i}(v)$ vertices of
color $i$ in its neighborhood. In a uniform random ordering $\pi$, the probability that $v$ appears before all $c_i(v)$ of these vertices is just $\frac{1}{c_{i}(v)+1}$. Thus, 
\[
T(c,v) \le k-\sum_{i\le k}\Big(1-\frac{1}{c_{i}(v)+1}\Big) = \sum_{i\le k}\frac{1}{c_{i}(v)+1} = W(c,v),
\]
as desired. 
\end{proof}

\section{Setup}\label{sec:Setup}

We outline the proof of Theorem \ref{thm:main}.
Our strategy is to control the size of a minimal counterexample
to Tomescu's conjecture (equivalently, to induct on the number of vertices $n$ in $G$).

Suppose Theorem \ref{thm:main} is false for a certain $k\ge 4$. Fix this $k$ henceforth. Let $n$ be the minimum number of vertices in a counterexample to Theorem~\ref{thm:main}, so that $n > k$. Note that such a counterexample $G$ either satisfies (\ref{eq:tomescu-conj}) strictly, or else 
\[
P_G(k) = k!(k-1)^{n-k}
\]
and the $2$-core of $G$ is not a $k$-clique.

Define a graph $G=(V,E)$ to be {\it bad} if $\chi(G)=k$, $|V|=n$, and $G$ is a minimal counterexample to Theorem~\ref{thm:main}, i.e. every proper subgraph of $G$ satisfies Theorem~\ref{thm:main}.

Our proof breaks into three parts. First, we use a straightforward edge-counting argument to show that there are no bad graphs for $n\le 2k-2$. Then, we show using the Overprediction Lemma that any bad graph with $k \ge 5$ satisfies $n\le 2k-2$. Finally, for $k = 4$ we modify the argument to show that $n\le 7$, which leaves exactly the case $k=4$, $n=7$ to check by hand.

Recall that a graph $G$ is {\it $k$-critical} if $\chi(G)=k$ and deleting any vertex or edge reduces its chromatic number.

\begin{lem}
Every bad graph is $k$-critical.
\begin{proof}
Suppose $G$ is bad but not $k$-critical, and there is an edge $e$ for which $G\backslash e$ is still $k$-chromatic. Since $G$ is a minimal counterexample to Theorem \ref{thm:main} with respect to edge deletion, if $G\backslash e$ is connected, then it satisfies
\[
P_{G\backslash e}(k) \le k!(k-1)^{n-k}.
\]

But every $k$-coloring of $G$ is a $k$-coloring of $G\backslash e$, so $G$ also satisfies Theorem \ref{thm:main}, which is a contradiction. 

If $G\backslash e$ is disconnected, let the two components be the induced subgraphs $G[V_1]$ and $G[V_2]$ for some vertex partition $V=V_1 \sqcup V_2$. Since $e$ is the only edge between these two components in $G$ and $\chi(G) = k$, one of the two components (say $G[V_1]$) has chromatic number $k$. Since $G[V_2]$ is connected, every $k$-coloring of $G[V_1]$ extends to at most $(k-1)^{|V_2|}$ $k$-colorings of $G$, with equality if and only if $V_2$ is a tree. Also, by the minimality of $G$,
\[
P_{G_1}(k) \le k!(k-1)^{|V_1| - k},
\]
with equality if and only if the $2$-core of $G[V_1]$ is a $k$-clique. Thus,
\[
P_G(k) \le P_{G_1}(k)\cdot(k-1)^{|V_2|} \le k!(k-1)^{n-k},
\]
with equality if and only if the $2$-core of $G[V_1]$ is a $k$-clique and $G[V_2]$ is a tree. But then the $2$-core of $G$ is a $k$-clique, so $G$ is not a counterexample to Theorem \ref{thm:main}. This is a contradiction and completes the proof. 
\end{proof}
\end{lem}

We will need some well-known structural results about $k$-critical graphs. It is a basic fact that the minimum degree of a $k$-critical graph $G$ is at least $k-1$. We will also need a classical result of Brooks. 

\begin{lem}[Brooks' Theorem \cite{Br}]
\label{lem:degk}  If $k\ge 4$, a $k$-critical graph that is not the $k$-clique has at least one vertex of degree at least $k$.
\end{lem}

The structure of $k$-critical graphs has been studied extensively, and much more is now known -- see for example Gallai \cite{Ga1, Ga2}, Dirac \cite{Di}, and Kostochka and Yancey \cite{KoYa1,KoYa2}.

We also use a general fact about $k$-colorings of graphs using the minimum possible number of colors. 

\begin{lem}
\label{lem:radiant} For every graph $G=(V,E)$ of chromatic number $k$ and for every proper $k$-coloring
$c$ of $G$, there are $k$ vertices $v_{1},\ldots,v_{k} \in V$ such that, for each $i\le k$, $c(v_{i})=i$ and $\{v_{i}\}\cup N(v_{i})$ contains a vertex
of every color. 
\end{lem}
\begin{proof}
Suppose there is no vertex of some color $i$ with every color in $[k]\backslash {i}$ appearing in its neighborhood.
Then, for every vertex $v$ with $c(v)=i$, there is some color $j\ne i$ not appearing in $N(v)$, and we can change the color of $v$ to $j$. Since the vertices of color $i$ in $k$-coloring $c$ form an independent set, these changes can be made simultaneously to create a proper $(k-1)$-coloring of $G$ not using color $i$. This shows $\chi(G)\le k-1$, which is a contradiction.
\end{proof}

Thus, to every $k$-coloring $c$ of $G$, we can associate a set
of $k$ vertices. There is one of each color, and each sees all the other colors
in its neighborhood. We call these $k$ vertices the \emph{radiant}
vertices of $k$-coloring $c$. When there are multiple vertices of the same color with this property, we arbitrarily fix a single one of them to call radiant. If a vertex $v$ is radiant in $k$-coloring $c$, then we say $v$ is \emph{$c$-radiant}.

Finally, for bad graphs specifically, we show that each pair of vertices receive the same color in less than a fraction $1/(k-1)$ of all $k$-colorings.
\begin{lem}
\label{lem:bad-probability}If $G$ is a bad graph, and
$u,v$ are any two distinct vertices of $G$, then 
\begin{equation*}
\Pr[c(u)=c(v)]<\frac{1}{k-1},
\end{equation*}
where the probability is over a uniform choice of a $k$-coloring
$c$ of $G$.
\end{lem}
\begin{proof}
Suppose this is not true for some pair $u,v$. Note that $u,v$ are not adjacent as otherwise $u$ and $v$ must receive different colors and $\Pr[c(u)=c(v)]=0$. Let $G'$ be the graph obtained from $G$ by contracting $u$ and $v$ together. That is, the two vertices $u$ and $v$ are replaced by a single vertex whose neighborhood is the union $N(u) \cup N(v)$. The $k$-colorings of $G'$ are in bijection with the $k$-colorings of $G$ that assign $u$ and $v$ the same color. Thus, 
\[ 
P_{G'}(k) = \Pr\big[c(u)=c(v)\big]P_{G}(k) \ge \frac{1}{k-1}P_{G}(k) \ge k!(k-1)^{|V|-k-1}
\]
since $G$ itself contradicts Theorem \ref{thm:main}. Note that $G'$ is connected because $G$ is connected. 

Thus, either the $2$-core of $G'$ is a $k$-clique or $G'$ is a counterexample to Theorem \ref{thm:main} with fewer vertices. The latter assumption contradicts the minimality of $G$, so the $2$-core of $G'$ is a $k$-clique.

Since $G$ is bad, $G$ is $k$-critical and has minimum degree at least $k-1\ge3$. But a single contraction reduces the degree of each vertex by at most $1$, so the minimum degree of $G'$ is at least $2$. Thus, $G'$ is its own $2$-core and must be exactly $K_k$. 

This is impossible. Indeed, it would have to be the case that $G\backslash \{u,v\}$ is a $(k-1)$-clique, and since $\chi(G)=k$ at least one of $u$ and $v$ must be complete to this $(k-1)$-clique. But then $G$ would contain a $k$-clique as a proper subgraph, and would not be $k$-critical.
\end{proof}

The next lemma is a straightforward consequence of Lemma \ref{lem:bad-probability}.   
\begin{lem}
\label{lem:subset-pairs} Let $G=(V,E)$ be a bad graph. For any subset $U\subset V$, if $P(c,U)$
is the number of pairs ${u,v}$ of distinct vertices in $U$ for which $c(u)=c(v)$, then 
\[
\mathbb{E}_{c}[P(c,U)]<\frac{1}{k-1}\binom{|U|}{2}.
\]
\end{lem}
\begin{proof}
For every pair $(u,v)$ of vertices in $U$, the probability they
are the same color in a random $k$-coloring of $G$ is less than $\frac{1}{k-1}$
by Lemma \ref{lem:bad-probability}. The stated inequality follows by linearity of expectation. 
\end{proof}

In general, we will apply Lemma \ref{lem:subset-pairs} when $U=V$ or when $U\subseteq N(v)$ for some vertex $v$ and $|U| = k-1$.

We describe an immediate application of Lemma \ref{lem:bad-probability}, showing there are no small bad graphs.

\begin{lem}\label{lem:edge-count}
If there exists a bad graph on $n$ vertices, then $n \ge 2k-1$.
\end{lem}
\begin{proof}
Suppose there is a bad graph $G=(V,E)$ with $\chi(G)=k\ge 4$ and $|V|=n\le 2k-2$.
Define $P(c)=P(c,V)$ to be the total number of unordered pairs $u\ne v \in V$ for which $c(u)=c(v)$.

Applying Lemma \ref{lem:bad-probability} to every pair $u \not\sim v$ of nonadjacent vertices in $G$ and summing over all such pairs, we find by linearity of expectation that
\begin{equation}\label{eq:pairs-upper}
\mathbb{E}_c[P(c)] < \frac{1}{k-1} \Big(\binom{n}{2} - |E|\Big).
\end{equation}

On the other hand, $P(c)=\sum_{i} \binom{c_i}{2}$, where $c_i$ is the size of the $i$-th color class in $k$-coloring $c$. As $\sum_i c_i = n \le 2k-2$ and each $c_i$ is a nonnegative integer, convexity of the binomial coefficient implies 
\begin{equation}
P(c) = \sum_{i\le k} \binom{c_i}{2} \ge (n-k) \binom{2}{2} + (2k-n)\binom{1}{2} = n-k. \label{eq:pairs-lower}
\end{equation}

Comparing (\ref{eq:pairs-upper}) to (\ref{eq:pairs-lower}), we have
\[
n-k < \frac{1}{k-1} \Big(\binom{n}{2} - |E|\Big),
\]
which simplifies to
\begin{equation}\label{eq:edge-bound}
|E| < \binom{n}{2} - (n-k)(k-1) = \frac{n^2+n}{2} + k^2 - kn - k.
\end{equation}

Since $G$ is bad, it has at least $k+1$ vertices. By an old result of Gallai \cite{Ga2}, the number of edges of any $k$-critical graph on $n$ vertices with $k < n\le 2k-1$ is at least
\begin{equation}\label{eq:gallai}
\frac{k-1}{2} n + \frac{(n-k)(2k-n)}{2} - 1.
\end{equation} 

Comparing (\ref{eq:edge-bound}) with (\ref{eq:gallai}), it follows that
\[
n^2 + (1-3k) n + (2k^2 - k + 1) > 0,
\]
and it is easy to check that this inequality does not hold whenever $k\ge 4$ and $n \le 2k-2$.
\end{proof}

\section{A Linear Program \label{sec:A-Linear-Program}}

We reduced upper bounding the number of $k$-colorings of $G$
to estimating the geometric mean of the quantity $W(c,v)$ defined by (\ref{eq:defW}), which depends
only on the integer partition $\deg(v)=c_{1}(v)+c_{2}(v)+\cdots+c_{k}(v)$
of $\deg(v)$ into the sizes of the color classes within $N(v)$. We now
solve a linear program that we will use to bound $W(c,v)$ in the next section.

\begin{lem}
\label{lem:lin-program}Suppose $x\ge k\ge4$ are positive integers,
$S\in[k-1,3k-3]$ is a real number, and $a_{j}$, for $j\ge0$, are nonnegative
real numbers for which 
\begin{align}
\sum_{j\ge 0}a_{j} & = x \label{eq:linprog-1}\\
\sum_{j\ge 0}ja_{j} & = k-1 \label{eq:linprog-2}\\
\sum_{j\ge 0}j^{2}a_{j} & \le S. \label{eq:linprog-3}
\end{align}

Then, 
\begin{equation}\label{eq:W-piecewise-max}
\sum_{j\ge 0}\frac{1}{j+1}a_{j}\le\begin{cases}
x-\frac{4(k-1)-S}{6} & \text{if }S\le2k-2\\
x-\frac{6(k-1)-S}{12} & \text{if }S>2k-2.
\end{cases}
\end{equation}
\end{lem}
\begin{proof}
Write $W = \sum_{j\ge 0} {\frac{1}{j+1} a_j}$.

The maximum value of $W$ will always be attained when (\ref{eq:linprog-3}) achieves equality. Otherwise, subtract $2\varepsilon$ from some nonzero $a_{j},j\ge1$, and add $\varepsilon$ to each of $a_{j-1},a_{j+1}$. For a small enough $\varepsilon >0$, this operation maintains the constraints (\ref{eq:linprog-1}), (\ref{eq:linprog-2}), and (\ref{eq:linprog-3}) while increasing $W$.

It is a standard fact that for a linear program with $n$ variables whose set of feasible solutions is nonempty and bounded, there exists an optimal solution for which some set of $n$ linearly independent constraints are exactly equality. Such solutions are called {\it basic feasible solutions} (see for example Theorems 2.3 and 2.7 of \cite{BeTs}). For the linear program (\ref{eq:linprog-1})-(\ref{eq:linprog-3}) in question, there are only three constraints other than the nonnegativity constraints $a_j \ge 0$. Thus, some optimal solution has at most three nonzero $a_j$. We assume henceforth that there are at most three distinct values of $j$ with $a_j \ne 0$.

Subtracting (\ref{eq:linprog-2}) from (\ref{eq:linprog-1}), we see that $a_{0}>0$, so there are at most
two nonzero $a_{j}$ with $j > 0$. Suppose $0<j_{0}<j_{1}$ are two indices such
that if $j\not\in\{0,j_{0},j_{1}\}$ then $a_{j}=0$. The three constraints are linearly independent and uniquely determine the values
of $a_{0},a_{j_{0}},a_{j_{1}}$:
\begin{align*}
a_{0}+a_{j_{0}}+a_{j_{1}} & = x\\
j_{0}a_{j_{0}}+j_{1}a_{j_{1}} & = k-1\\
j_{0}^{2}a_{j_{0}}+j_{1}^{2}a_{j_{1}} & = S.
\end{align*}

The unique solution to this system is

\[
(a_0, a_{j_0}, a_{j_1}) = \Big(x-a_{j_{0}}-a_{j_{1}}, \frac{(k-1)j_{1}-S}{j_{0}j_{1}-j_{0}^{2}}, \frac{S-(k-1)j_{0}}{j_{1}^{2}-j_{0}j_{1}}\Big).
\]

As we require $a_{j_0}, a_{j_1} \ge 0$, the solution exists whenever $j_{0}\le\frac{S}{k-1}\le j_{1}$. In
terms of $j_{0},j_{1}$, the final objective function is 
\begin{align}
W & = a_{0}+\frac{1}{j_{0}+1}a_{j_{0}}+\frac{1}{j_{1}+1}a_{j_{1}} \nonumber\\
 & = x-\frac{j_{0}}{j_{0}+1}a_{j_{0}}-\frac{j_{1}}{j_{1}+1}a_{j_{1}} \nonumber\\
 & = x-\frac{(k-1)(j_{0}+j_{1}+1)-S}{(j_{0}+1)(j_{1}+1)}. \label{eq:maxW}
\end{align}

We see that $W$ is at most the maximum value of (\ref{eq:maxW}) over
all choices of integers $0<j_{0}<j_{1}$ for which $j_{0}\le\frac{S}{k-1}\le j_{1}$.
It is easy to check that when $\frac{S}{k-1}\in[1,2]$, the best choice
is $j_{0}=1,j_{1}=2$, and if $\frac{S}{k-1}\in[2,3]$, the best choice
is $j_{1}=2,j_{2}=3.$ Plugging these choices into (\ref{eq:maxW}), we obtain exactly the desired bound (\ref{eq:W-piecewise-max}).
\end{proof}

Although we will only use the case $x=k$, we state Lemma \ref{lem:lin-program} for general $x$ because it may be useful to tackle the problem of counting $x$-colorings of graphs with chromatic number $k$ even when $x > k$.

\section{Main Theorem \label{sec:Radiant-Vertices}}
We are almost ready to apply the Overprediction Lemma to bad graphs. We make one final simplification which allows us to treat $G$ as if it is $(k-1)$-regular.

Suppose $G=(V,E)$ is a bad graph, so it has minimum degree at least $k-1$. For each vertex $v\in V$, fix an arbitrary set $N^*(v) \subseteq N(v)$ of exactly $k-1$ neighbors. This specified set of neighbors will take the place of $N(v)$ in some of the proceeding calculations. Define $N^{*-}_{\pi}(v) = N^{-}_{\pi}(v) \cap N^*(v)$, and $X_{\pi}^*(c,v)$ to be the number of distinct colors of $k$-coloring $c$ missing from $N^{*-}_{\pi}(v)$. 

Similarly, define $c_i^*(v)$ to be the number of vertices of color $i$ in $N^*(v)$, 
\[
T^*(c,v) = \exp(\mathbb{E}_{\pi}[\log X_{\pi}^*(c,v)])
\]
and 
\[
W^*(c,v) = \sum_{i\le k} \frac{1}{c_i^*(v)+1}.
\]

Then, it follows that $X_{\pi}(c,v) \le X_{\pi}^*(c,v)$, $c_i(v) \ge c_i^*(v)$, $T(c,v) \le T^*(c,v)$, and $W(c,v) \le W^*(c,v)$. It will suffice to bound the quantity $W^*(c,v)$ in place of $W(c,v)$.

\begin{thm}
\label{thm:large-n}If there exists a bad graph on $n$ vertices,
then either $k=4$ and $n\le 12$, or $k\ge 5$ and $n \le 2k-2$.
\end{thm}
\begin{proof}
Let $G$ be a bad graph. By applying
Lemma \ref{lem:subset-pairs} to $N^*(v)$ and averaging over all $v$, 
\begin{equation}
\mathbb{E}_{c, v}\Big[\sum_{i\le k}\binom{c_{i}^*(v)}{2}\Big] < \frac{1}{k-1}\binom{k-1}{2}
 = \frac{k-2}{2}. \label{eq:pairbound}
\end{equation}

It is possible to bound $T(c,v)$ on average with (\ref{eq:pairbound}), but we can obtain a stronger bound on $T(c,v)$ for the radiant vertices. Thus, we bound $T(c,v)$ when $v$ is $c$-radiant and when $v$ is not $c$-radiant separately.

If $v$ is $c$-radiant, let $u_1,\ldots, u_{k-1}$ be $k-1$ neighbors of $v$ representing every color except $c(v)$. In a uniform random ordering of $V$, the number of $u_i$ which appear before $v$ is uniformly distributed among $0,\ldots, k-1$. Thus, the probability that $N^{-}_{\pi}(v)$ contains vertices of at least $i$ different colors is at least $\frac{i}{k}$. Under these conditions,
\begin{equation}
T(c,v) = \exp\mathbb{E}_\pi [\log X_\pi(c,v)] = \prod_{i\le k-1} (k-i)^{\Pr[X_\pi(c,v) = k-i]} \le (k!)^{\frac{1}{k}}. \label{eq:radiant-T}
\end{equation}

This is our bound for $T(c,v)$ when $v$ is $c$-radiant. Now we seek to control $T(c,v)$ when $v$ is not $c$-radiant using (\ref{eq:pairbound}). Define
\[
\mathbb{E}^*_{c, v}[Y] = \mathbb{E}_c[\mathbb{E}_v[Y|\text{$v$ is not $c$-radiant}]]
\]
to be the expectation of $Y=Y(c,v)$ under a uniform random choice of a pair $(c,v)$ of a $k$-coloring $c$ and a vertex $v$ which is not $c$-radiant.

In a uniform random choice of a pair $(c,v)$ of a $k$-coloring and a vertex, the probability that $v$ is $c$-radiant is $\frac{k}{n}$. Since the sum of $\binom{c_i^*(v)}{2}$ is always nonnegative, this observation together with (\ref{eq:pairbound}) implies
\begin{equation}
\mathbb{E}^*_{c,v}\Big[\sum_{i\le k}\binom{c_{i}^*(v)}{2}\Big] \le \Big(1-\frac{k}{n}\Big)^{-1} \mathbb{E}_{c,v}\Big[\sum_{i\le k}\binom{c_{i}^*(v)}{2}\Big]
< \frac{n(k-2)}{2(n-k)}. \label{eq:pairbound-conditional}
\end{equation}

For $j \ge 0$, let $a_{j}=\mathbb{E}^*_{c,v}[|\{i:c_{i}^*(v)=j\}|]$.
Then, we claim that these $a_j$ satisfy the conditions
\begin{align}
\sum_{j\ge 0}a_{j} & = k \label{eq:aj1}\\
\sum_{j\ge 0}ja_{j} & = k-1 \label{eq:aj2}\\
\sum_{j\ge 0}j^{2}a_{j} & < S \label{eq:aj3}
\end{align}
of Lemma \ref{lem:lin-program} with
\[
S=(k-1)+\frac{n(k-2)}{n-k}.
\]

The first constraint (\ref{eq:aj1}) is the fact that there are $k$ colors in total. The second constraint (\ref{eq:aj2}) comes from the fact that $|N^*(v)|=k-1$ for each $v$. Adding (\ref{eq:aj2}) to twice (\ref{eq:pairbound-conditional}), we get the third constraint (\ref{eq:aj3}).

The corresponding objective function $\sum_{j \geq 0} \frac{1}{j+1} a_j$ in Lemma \ref{lem:lin-program} is exactly $\mathbb{E}^*_{c,v}[W^*(c,v)]$. We apply Lemma \ref{lem:lin-program} and obtain 
\begin{align*}
\mathbb{E}^*_{c,v}[W^*(c,v)] \le  \begin{cases}
\frac{7k+5}{12}+\frac{n(k-2)}{12(n-k)} & \text{if }2k-1\le n\le k^{2}-k\\
\frac{k+1}{2}+\frac{n(k-2)}{6(n-k)} & \text{if }n > k^{2}-k.
\end{cases}
\end{align*}

From Lemma \ref{lem:AMGM}, $T(c,v)\le W(c,v)$, and we noted that $W(c,v) \le W^*(c,v)$. Hence, by the Overprediction Lemma,
\begin{align*}
P_G(k) &\le 
\prod_{v\in V}\exp(\mathbb{E}_{c}[\log T(c,v)]) \\
& = \exp(n\mathbb{E}_{c,v}[\log T(c,v)]) \\
& = \exp(k\mathbb{E}_{c,v}[\log T(c,v) | \text{$v$ is $c$-radiant}] + (n-k) \mathbb{E}^*_{c,v}[\log T(c,v)]) \\
& \le k! \mathbb{E}^*_{c,v}[W^*(c,v)]^{n-k} \\
& \le
\begin{cases}
k!\Big(\frac{7k+5}{12}+\frac{n(k-2)}{12(n-k)}\Big)^{n-k} & \text{if }2k-1\le n\le k^{2}-k\\
k!\Big(\frac{k+1}{2}+\frac{n(k-2)}{6(n-k)}\Big)^{n-k} & \text{if }n > k^{2}-k.
\end{cases}
\end{align*}

The factor of $k!$ comes from (\ref{eq:radiant-T}). Note that
\[
\frac{7k+5}{12}+\frac{n(k-2)}{12(n-k)} < k-1 \\
\]
when $k\ge 5$ and $2k-1 \le n \le k^2-k$, while
\[
\frac{k+1}{2}+\frac{n(k-2)}{6(n-k)} < k-1
\]
when $k\ge 5$ and $n > k^2 - k$ or $k=4$ and $n\ge 13$.

Plugging these bounds in, we get
\[
P_G(k) < k!(k-1)^{n-k}
\]
when $k\ge5$ and $n\ge2k-1$ or $k=4$ and $n\ge 13$. This completes the proof.
\end{proof}

Theorem \ref{thm:large-n} and Lemma \ref{lem:edge-count} together show that there are no bad graphs except possibly in the case $k=4$ and $7\le n \le 12$. In the next section we improve the above argument to handle $8 \le n \le 12$ and then check $n=7$ by hand.

\section{Small $4$-Critical Graphs\label{sec:small-4-crit}}

In our application of the Overprediction Lemma in Theorem \ref{thm:large-n}, we bounded $T(c,v)$ by the larger but simpler quantity $W(c,v)$ via the AM--GM Inequality. When $k=4$ it is possible to directly bound the average value of $T(c,v)$, improving Theorem \ref{thm:large-n}. We use the same notations as in Section~\ref{sec:Radiant-Vertices}.

\begin{lem} \label{lem:small-4-crit}
If there exists a bad graph on $n$ vertices with chromatic number $k=4$, then $n\le 7$.
\begin{proof}
Let $G$ be the bad graph in question. We use a very similar argument to the proof of Theorem \ref{thm:large-n}. Inequalities (\ref{eq:radiant-T}) and (\ref{eq:pairbound-conditional}) still hold for $G$, showing (respectively) that if $v$ is $c$-radiant,
\begin{equation}
T(c,v) \le (4!)^{\frac{1}{4}}, \label{eq:radiant-T-4}
\end{equation}
and that conditional on $v$ being non-radiant,
\begin{equation}
\mathbb{E}^*_{c,v}\Big[\sum_{i\le 4}\binom{c_{i}^*(v)}{2}\Big] < \frac{n}{n-4}. \label{eq:pairbound-conditional-4}
\end{equation}
Recall that $\mathbb{E}^*_{c,v}$ is the expectation over a uniform choice of a coloring $c$ and a vertex $v$ which is not $c$-radiant.

Now, observe that $|N^*(v)| = k-1 = 3$ for every vertex $v$. For each $j \in {1,2,3}$, define $S_j$ to be the set of non-radiant vertices $v$ for which $N^*(v)$ contains vertices of exactly $j$ distinct colors. The exact value of $T^*(c,v)$ is determined as follows:
\[
T^*(c,v) = \begin{cases}
4^{\frac{1}{4}}\cdot 3^{\frac{3}{4}} &  \text{if } v\in S_1\\
4^{\frac{1}{4}} \cdot 3^{\frac{1}{3}} \cdot 2^{\frac{5}{12}} & \text{if } v \in S_2 \\
(4!)^{\frac{1}{4}} & \text{if } v\in S_3.
\end{cases}
\]
The first case, for example, comes from the fact that if the three elements of $N^*(v)$ are all the same color, then there is a $\frac{1}{4}$ chance that $v$ precedes all of $N^*(v)$ in the random ordering $\pi$ so that $X_{\pi}(c,v)=4$, and otherwise $X_{\pi}(c,v) = 3$.

Let $s_j = \mathbb{E}_{c}[|S_j|]$. Conditioning on the value of $j$ for which $v\in S_j$, we are left with the problem of upper bounding
\begin{equation}\label{eq:exact-T-4}
(n-4)\mathbb{E}_{c,v}^*[\log T^*(c,v)] = s_1 \log(4^{\frac{1}{4}}\cdot 3^{\frac{3}{4}}) + s_2 \log(4^{\frac{1}{4}} \cdot 3^{\frac{1}{3}} \cdot 2^{\frac{5}{12}}) + s_3 \log((4!)^{\frac{1}{4}}).
\end{equation}

Using the fact that there are $4$ radiant vertices, we see that
\begin{equation}
s_1 + s_2 + s_3 = n-4. \label{eq:sj1}
\end{equation}

Next, consider the size of 
\[
P(v)=\sum_{i\le 4}\binom{c_{i}^*(v)}{2}
\]
when $v \in S_1, S_2,$ and $S_3$. When $v\in S_1$, there is a single color for which $c_i^*(v) = 3$, so $P(v) = 3$. When $v\in S_2$, the sizes of $c_i^*(v)$ are $2,1,0,0$, so $P(v) = 1$. When $v\in S_3$, the sizes of $c_i^*(v)$ are $1,1,1,0$, so $P(v) = 0$. Thus, by breaking up the expecation in (\ref{eq:pairbound-conditional-4}) conditional on which $S_j$ each vertex lies in, we get
\begin{equation}
3s_1 + s_2 = (n-4)\mathbb{E}^*_{c,v}\Big[\sum_{i\le 4}\binom{c_{i}^*(v)}{2}\Big] \le n. \label{eq:sj2}
\end{equation}

It remains to solve another linear program. Assuming $n\ge 6$, we find that $(s_1, s_2, s_3) = (2, n-6,0)$ is the unique triple maximizing the quantity (\ref{eq:exact-T-4}) subject to $s_j\ge 0$ and the two constraints (\ref{eq:sj1}) and (\ref{eq:sj2}).

Putting all this together with the Overprediction Lemma,
\begin{align*}
P_G(4) &\le \exp(n\mathbb{E}_{c,v}[\log T(c,v)]) \\
& = \exp(4\mathbb{E}_{c,v}[\log T(c,v) | \text{$v$ is $c$-radiant}] + (n-4) \mathbb{E}^*_{c,v}[\log T(c,v)]) \\
& \le 4! \mathbb{E}^*_{c,v}[T^*(c,v)]^{n-4} \\
& \le 4! \cdot (4^{\frac{1}{4}}\cdot 3^{\frac{3}{4}})^2 \cdot (4^{\frac{1}{4}} \cdot 3^{\frac{1}{3}} \cdot 2^{\frac{5}{12}})^{n-6} \\
& = 4! \cdot 3^{(2n-3)/6} \cdot 2^{(11n-54)/12}.
\end{align*}

The factor of $4!$ comes from (\ref{eq:radiant-T-4}). The last quantity above is smaller than $4! \cdot 3^{n-4}$ if $n\ge 8$, so there are no bad graphs when $n\ge 8$, as desired.
\end{proof}
\end{lem}

We can now finish the proof of Tomescu's conjecture.

{\it Proof of Theorem \ref{thm:main}.}
Combining Theorem \ref{thm:large-n} with Lemma \ref{lem:edge-count}, we see that there are no bad graphs except possibly if $k=4$ and $7\le n \le 12$. In this case, Lemma \ref{lem:small-4-crit} rules out $n \ge 8$, which leaves the sole possibility of $k=4$ and $n = 7$.

Up to isomorphism, there are exactly two $4$-critical graphs on $7$ vertices. One of them is the so-called Mycielskian of a triangle, which can be described explicitly as the graph on vertices $u_1,u_2,u_3, v_1, v_2, v_3, w$, where the edges are $u_i \sim u_j$ for all $i\ne j$, $v_i \sim u_j$ for all $i \ne j$, and $w \sim v_i$ for all $i$. This graph has $4!\cdot 13$ colorings by $4$ colors.

The second $4$-critical graph is the Moser spindle, which is obtained by gluing two diamonds at a vertex and connecting their opposite vertices. Explicitly, it is the graph on vertices $u_1, u_2, u_3, v_1, v_2, v_3, w$ where $u_i \sim u_j$ if $i\ne j$, $v_i~ v_j$ if $i\ne j$, $u_1 \sim v_1$, and $w$ is adjacent to $u_2, u_3, v_2, v_3$. This graph has $4!\cdot 16$ colorings by $4$ colors.

Both graphs satisfy $P_G(4)\le 4!\cdot 3^3$, so we are done. 
\hfill\qed

\section{Concluding Remarks \label{sec:Closing-Remarks}}

In 1990, Tomescu \cite{To1} extended his original conjecture to a general number of colors.

\begin{conjecture}
If $x\ge k\ge4$, and $G$ is a graph on $n$ vertices with chromatic number $k$,
then
\begin{equation*}
P_{G}(x)\le(x)_{k}(x-1)^{n-k},
\end{equation*}
with equality if and only if the $2$-core of $G$ is a $k$-clique.
\end{conjecture}
Here $(x)_{k}$ is the falling factorial $x(x-1)\cdots(x-k+1)$. Knox
and Mohar \cite{KnMo1,KnMo2} proved this more general conjecture for
$k=4$ and $5$, but our methods meet an obstruction to resolving the general
case $x>k$. We relied on the existence of radiant vertices
in Lemma \ref{lem:radiant} to recover the factor of $k!$ when $x=k$, and such vertices do not necessarily exist in $x$-colorings of graphs of chromatic number $k$ when $x>k$.

Our proof of Tomescu's conjecture began with the fact that a minimal counterexample would have to be $k$-critical. However, while the bound in Tomescu's conjecture is tight for chromatic number $k$, it is no longer tight when we restrict our attention to $k$-critical graphs. We are thus led to the following question.

\begin{ques}\label{criticalquestion}
What is the maximum of $P_G(k)$ over all $k$-critical graphs on $n$ vertices?
\end{ques}
Knox and Mohar \cite{KnMo1} recently announced the bound $P_G(k) \leq (k-1)^{n-c\log n}(k-2)^{c\log n}$ holds for every graph $G$ on $n$ vertices of minimum degree at least three and with no twins (two vertices with identical neighborhoods), where $c>0$ is an absolute constant. This same bound holds for every $k$-critical graph with $k \geq 4$ because critical graphs have no twins, and improves on Theorem \ref{thm:main} when $n$ is superexponentially large in $k^2$. Our methods suggest much stronger bounds than Theorem \ref{thm:main} may hold in Question \ref{criticalquestion} and we plan to explore this problem further.

There are many other attractive optimization problems in the study of graph colorings. Linial \cite{Li} studied the problem of minimizing $P_{G}(k)$ over all graphs with $n$ vertices and $m$ edges, and found that the complete graph $K_s$ with one additional vertex adjacent to $l$ vertices of the clique and $n-s-1$ isolated vertices, where $m={s \choose 2}+l$ and $0 \leq l < s$, is a minimizer for every integer $k$.

Linial also proposed the related problem of maximizing $P_{G}(k)$ over graphs $G$ with $n$ vertices and $m$ edges, which is more challenging because the extremal graphs no longer have such a simple structure. Bender and Wilf \cite{BeWi,Wi} arrived at the same maximization problem independently, via the study of the backtrack algorithm for enumerating all proper $k$-colorings of a graph $G$. Lazebnik \cite{La} was the first to make substantial progress on this problem, and conjectured that the Tur\'an graphs are always extremal graphs for the coloring-maximization problem (see \cite{LaPiWo}). Many cases of Lazebnik's conjecture have been verified \cite{LaTo, LoPiSu, No, Tof}, but the full conjecture was recently disproved by Ma and Naves \cite{MaNa}.

More recently, Galvin and Tetali \cite{GaTe} considered the problem of maximizing $P_G(k)$ over $d$-regular graphs $G$ with $n$ vertices. They conjectured that for $n$ a multiple of $2d$, the optimal $G$ consists of $\frac{n}{2d}$ disjoint copies of the complete bipartite graph $K_{d,d}$. Using entropy methods, they proved this conjecture under the additional assumption that $G$ is bipartite. This conjecture was later resolved using a linear programming relaxation for $3$-regular graphs by Davies et al. \cite{DaJePeRo} and for $4$-regular graphs when $k\ge 5$ by Davies \cite{Da}.

\medskip

\noindent {\bf Acknowledgements.} We would like to thank Aysel Erey, J\'anos Pach and the referees for many helpful comments.

\medskip

After we submitted this paper, we learned \cite{Mo} that Bojan Mohar and Fiachra Knox have also announced a proof of Tomescu's conjecture. We also learned from Persi Diaconis that the counting method in the Overprediction Lemma can be viewed as an instance of a general estimation method known as sequential importance sampling (see for example \cite{BlDi}).

\end{document}